\documentclass[12pt]{article}
\usepackage{amssymb,amsfonts,amsmath,amsthm,comment,dsfont,xcolor}
\usepackage[colorlinks, citecolor=purple,linkcolor=blue,urlcolor=green]{hyperref}
\usepackage[lite,alphabetic]{amsrefs}
\usepackage[a4paper, portrait, margin=23mm]{geometry}

\newcommand{\1}{\mathds{1}}

\newcommand{\N}{\mathbb{N}}

\newcommand{\Bo}{\mathrm{B}}

\newcommand{\8}{\infty}

\newcommand{\supp}{\mathrm{supp}}
\newcommand{\Co}{\mathcal{C}}

\newcommand{\Po}{\mathcal{P}}

\newcommand{\Io}{\mathcal{I}}
\newcommand{\Jo}{\mathcal{J}}

\newcommand{\toru}{\xrightarrow[]{\mathrm{ru}}}

\newcounter{dummy} 
\newtheorem{theorem}[dummy]{Theorem}

\newtheorem{proposition}[dummy]{Proposition}

\theoremstyle{remark}
\newtheorem{remark}[dummy]{Remark}

\begin{document}

\title{Normed lattices majorizing in their norm completions}

\author{Eugene Bilokopytov\thanks{Instituto de Ciencias Matem\'aticas, Madrid, Spain (\texttt{bilokopy@ualberta.ca}). This research has been partially supported by grant PCI2024-155094-2 funded by MICIU/AEI/10.13039/50110001103}~ and Viktor Bohdanskyi\thanks{Department of Analysis and Probability, Ukrainian National Technical University ``Igor Sikorsky Kyiv Polytechnic Institute'' (\texttt{viktorbohdanskyi@gmail.com}).}}
\maketitle

\begin{abstract}
This note is a follow-up to \cite{bt}. We focus on conditions under which a normed lattice $X$ is majorizing in its norm completion. We show that \cite[Question 8.17]{bt} -- namely, whether this holds whenever every norm-null sequence in $X$ has an order-bounded subsequence -- is equivalent to the question whether every P-ideal on $\N$ is meager. This is a longstanding open problem in Set Theory, and it has a negative answer under various set-theoretical assumptions, in particular under the Continuum Hypothesis.

We also present several equivalent conditions to both of the two aforementioned properties, and give a simple proof of a well-known Riesz-Fischer-style characterization of completeness of a normed lattice.\medskip


MSC2020 46A19, 46A40, 46B42, 54G99, 03E05, 03E35.
\end{abstract}\bigskip

We start with a simplified unified proof of \cite[Theorems 5.14 and 5.15]{ab0}. The authors would like to thank Vladimir Troitsky for some ideas that were used in this proof. The condition (v) below is called the \emph{Riesz-Fischer property}.

\begin{theorem}\label{rf}
For a normed lattice $X$ the following conditions are equivalent:
\item[(i)] The norm on $X$ is complete, that is, $X$ is a Banach lattice;
\item[(ii)] Every increasing norm-Cauchy sequence in $X$ converges;
\item[(iii)] Every increasing norm-Cauchy sequence in $X$ has a supremum;
\item[(iv)] If $\left(y_{n}\right)_{n\in\N}\subseteq X_{+}$, and $\sum\limits_{n\in\N}\|y_{n}\|<\8$, then there is $\bigvee\limits_{n\in\N} \sum\limits_{k=1}^{n}y_{k}$;
\item[(v)] If $\left(y_{n}\right)_{n\in\N}\subseteq X_{+}$, and $\sum\limits_{n\in\N}\|y_{n}\|\le 1$, then there is $\bigvee\limits_{n\in\N} \sum\limits_{k=1}^{n}y_{k}$, which is an element of $\Bo_{X}$.
\end{theorem}
\begin{proof}
(i)$\Rightarrow$(ii) and (v)$\Rightarrow$(iv) are trivial. (ii)$\Rightarrow$(iii) follows from \cite[Theorem 2.21(c)]{ab0}. To get (iii)$\Rightarrow$(iv) it is enough to consider $x_{n}:=\sum\limits_{k=1}^{n}y_{k}$. (i)$\Rightarrow$(v) follows from the triangle inequality, continuity of the norm and \cite[Theorem 2.21(c)]{ab0}.\medskip

(iv)$\Rightarrow$(i): \textbf{Claim.} If $\left(y_{n}\right)_{n\in\N}\subseteq X_{+}$ is such that $\sum\limits_{n\in\N} n\|y_{n}\|<\8$, and $z:=\bigvee\limits_{n\in\N} \sum\limits_{k=1}^{n}y_{k}$ (which exists by the assumption), then $z=\sum\limits_{n\in\N}y_{n}$.

Let $u:=\bigvee\limits_{n\in\N} \sum\limits_{k=1}^{n}ky_{k}$ (which exists by the assumption). Fix $m\in\N$ and define $z_{m}:=\bigvee\limits_{n\ge m} \sum\limits_{k=m}^{n}y_{k}$ (which also exists by the assumption). For every $n\ge m$ we have that $m\sum\limits_{k=m}^{n}y_{k}\le \sum\limits_{k=m}^{n}ky_{k}\le u$, and so $z_{m}\le \frac{1}{m}u$. For every $m\in\N$ distributivity of the sum yields $$z=\bigvee\limits_{n\ge m} \sum\limits_{k=1}^{n}y_{k}=\bigvee\limits_{n\ge m} \left(\sum\limits_{k=1}^{m-1}y_{k}+\sum\limits_{k=m}^{n}y_{k}\right)=z_{m}+\sum\limits_{k=1}^{m-1}y_{k},$$ and so $0\le z-\sum\limits_{k=1}^{m-1}y_{k}=z_{m}\le \frac{1}{m}u$. It follows that $\left\|z-\sum\limits_{k=1}^{m-1}y_{k}\right\|\le\frac{1}{m}\|u\|\to 0$, thus justifying $z=\sum\limits_{n\in\N}y_{n}$.\medskip

Now let $\left(x_{n}\right)_{n\in\N}\subseteq X$ be a norm-Cauchy sequence. In order to show that it has a limit it is enough to find a norm-convergent subsequence. Hence, without loss of generality we may assume that $\|x_{n}-x_{n-1}\|\le\frac{1}{n2^{n}}$, for every $n\in\N$ (we count $x_{0}:=0$). For $n\in\N$ let $v_{n}:=\left(x_{n}-x_{n-1}\right)^{+}$ and $w_{n}:=\left(x_{n}-x_{n-1}\right)^{-}$; note that $\|v_{n}\|,\|w_{n}\|\le \left\|\left|x_{n}-x_{n-1}\right|\right\|=\|x_{n}-x_{n-1}\|\le\frac{1}{n2^{n}}$.

Since $\left(v_{n}\right)_{n\in\N}$ and $\left(w_{n}\right)_{n\in\N}$ satisfy the assumptions of the claim, $v:=\sum\limits_{n\in\N}v_{n}$ and $w:=\sum\limits_{n\in\N}w_{n}$ exist in $X$. Note that for every $m\in\N$ we have $\sum\limits_{n=1}^{m}\left(v_{n}-w_{n}\right)=\sum\limits_{n=1}^{m}\left(x_{n}-x_{n-1}\right)=x_{m}$. Hence, $x_{m}\to v-w$.
\end{proof}

For $u\in X_{+}$ the principal ideal $I_{u}$ is endowed with the \emph{order unit norm} $\|\cdot\|_{u}$ defined by $\|x\|_{u}=\bigwedge\left\{r\ge 0:~ \left|x\right|\le r u\right\}$. A net $\left(x_{\alpha}\right)_{\alpha\in A}\subseteq X$ converges \emph{uniformly} to $x$ \emph{relative} to $u\in X_{+}$ if $x\in I_{u}$ and for every $\varepsilon>0$ there is $\alpha_{0}$ such that $\|x-x_{\alpha}\|_{u}\le\varepsilon$ (that is, $\left|x-x_{\alpha}\right|\le\varepsilon u$, and in particular $x_{\alpha}\in I_{u}$), for every $\alpha\ge \alpha_{0}$. We denote it by $x_{\alpha}\xrightarrow[]{\|\cdot\|_{u}}x$ and say that $\left(x_{\alpha}\right)_{\alpha\in A}\subseteq X$ converges \emph{relatively uniformly (ru)} to $x$, if $x_{\alpha}\xrightarrow[]{\|\cdot\|_{u}}x$, for some $u$. See \cite{bt} for more information about relative uniform convergence.\medskip

Let us add more equivalent conditions to \cite[Theorem 8.15]{bt}. Note that the conditions (iii)-(v) below are weakened versions of the corresponding conditions in Theorem \ref{rf}.

\begin{proposition}\label{arf}
For a normed lattice $X$ the following conditions are equivalent:
\item[(i)] $X$ is majorizing in its norm completion $\widetilde{X}$;
\item[(ii)] Every norm-Cauchy sequence in $X$ has an order bounded subsequence;
\item[(iii)] Every increasing norm-Cauchy sequence in $X$ is order bounded;
\item[(iv)] Every norm-summable sequence in $X$ is order bounded;
\item[(v)] If $\left(y_{n}\right)_{n\in\N}\subseteq X_{+}$, and $\sum\limits_{n\in\N}\|y_{n}\|< 1$, then there is $x\in \Bo_{X}$ such that $x\ge \sum\limits_{k=1}^{n}y_{k}$, for every $n\in\N$;
\item[(vi)] Every element of $\widetilde{X}$ is a relative uniform limit in $\widetilde{X}$ of a sequence in $X$ with a regulator in $X$.
\end{proposition}
\begin{proof}
Equivalence of (i)-(iii) and (vi) was proven in \cite[Theorem 8.15]{bt}. (v)$\Rightarrow$(iv) is trivial.

(iv)$\Rightarrow$(iii): Let $\left(y_{n}\right)_{n\in\N}\subseteq X$ be an increasing norm-Cauchy sequence. Note that it is enough to show that it has an order bounded subsequence, and so we may assume that $\|x_{n}\|\le \frac{1}{4^{n}}$, where $x_{n}:=y_{n}-y_{n-1}$, for $n\in\N$ (we count $y_{0}:=0$). Then, $\left(2^{n}x_{n}\right)_{n\in\N}$ is norm-summable, and so it has an upper bound $y$, so that $x_{n}\le\frac{1}{2^{n}}y$, for every $n\in\N$. Thus, $y_{n}=\sum\limits_{k=1}^{n}x_{k}\le \sum\limits_{k=1}^{n}\frac{1}{2^{k}}y\le y$, for every $n\in\N$.\medskip

(vi)$\Rightarrow$(v): Let $y:=\sum\limits_{n\in\N}y_{n}$ in $\widetilde{X}$. Note that $y\ge 0$ and $\|y\|\le \sum\limits_{n\in\N}\|y_{n}\|< 1$. By assumption there is $u\in X_{+}$ and $\left(x_{n}\right)_{n\in\N}\subseteq X$ such that $\left|y-x_{n}\right|\le\frac{1}{n}u$, for every $n\in\N$. Let $m\in\N$ be such that $\frac{2}{m}\|u\|\le 1-\|y\|$. We claim that $x_{m}+\frac{1}{m}u\in X$ satisfies the requirements. Indeed, for every $n\in\N$ we have $\sum\limits_{k=1}^{n}y_{k}\le y\le x_{m}+\frac{1}{m}u$. On the other hand, $0\le y\le x_{m}+\frac{1}{m}u\le y + \frac{2}{m}u$, which yields $\|x_{m}+\frac{1}{m}u\|\le \left\|y + \frac{2}{m}u\right\|\le \|y\|+\frac{2}{m}\|u\|\le 1$.
\end{proof}

Recall that a set $C\subseteq X$ is \emph{relatively uniformly closed} if every existing ru-limit of a net in $C$ is contained in $C$ (this condition is enough to verify for sequences). Since every ru-convergent sequence norm converges to the same limit, it follows that every norm-closed set (and, in particular, every finite set) is ru-closed. Let us discuss when the two classes of sets coincide. The condition (i) below is a weakened version of the conditions (ii) and (iv) in Proposition \ref{arf}.

\begin{proposition}\label{arfa}
For a normed lattice $X$ the following conditions are equivalent:
\item[(i)] Every norm-null sequence in $X$ has an order bounded subsequence;
\item[(ii)] Every norm-null net in $X$ contains (as a set) a relatively uniformly null sequence;
\item[(iii)] Every relatively uniformly closed set in $X$ is norm-closed.
\end{proposition}
\begin{proof}
(i)$\Rightarrow$(ii): Let $\left(x_{\alpha}\right)_{\alpha\in A}\subseteq X$ be norm-null. There are indices $\left(\alpha_{n}\right)_{n\in\N}$ such that $\|x_{\alpha_{n}}\|\le \frac{1}{n2^{n}}$, for every $n\in\N$. Then, $n\left|x_{\alpha_{n}}\right|\to 0$, and so by assumption there is an infinite $M\subseteq\N$ and $y\in X_{+}$ such that $n\left|x_{\alpha_{n}}\right|\le y$, for every $n\in M$. Hence, $\left(x_{\alpha_{n}}\right)_{n\in M}$ converges to $0$ uniformly relative to $y$, and is contained in $\left\{x_{\alpha}\right\}_{\alpha\in A}$.\medskip

(ii)$\Rightarrow$(iii): Let $C\subseteq X$ be relatively uniformly closed. Let $\left(x_{\alpha}\right)_{\alpha\in A}\subseteq C$ be norm-convergent. Without loss of generality we may assume that $x_{\alpha}\to 0$. Then, there are indices $\left(\alpha_{n}\right)_{n\in\N}$ such that $C\ni x_{\alpha_{n}}\toru 0$. As $C$ was assumed to be ru-closed, it follows that $0\in C$. Thus, $C$ contains all of its norm-limits, and so it is norm closed.\medskip

(iii)$\Rightarrow$(i): Let $\left(x_{n}\right)_{n\in\N}$ be a norm-null sequence, which has no order bounded subsequences. In particular, every element appears only finitely many times among $x_{n}$'s, and so by passing to a tail we may assume that $x_{n}\ne 0$, for every $n\in\N$. Every order bounded subset of $C:=\left\{x_{n}\right\}_{n\in\N}$ is finite. Since every ru-convergent sequence is order bounded, it follows that $C$ is ru-closed. By assumption, $C$ is also norm-closed. As $C\ni x_{n}\to 0$, it follows that $0\in C$, contradicting the previous assumption.
\end{proof}

In the following result the condition (i) is a simultaneous strengthening of the conditions (i) and (ii) of Proposition \ref{arfa}. It turns out that such a requirement is very restrictive.

\begin{proposition}\label{rfa}
For a normed lattice $X$ the following conditions are equivalent:
\item[(i)] Every norm-null net in $X$ has an order bounded subnet;
\item[(ii)] Every norm-null net in $X$ has an order bounded tail;
\item[(iii)] $X$ has a strong unit, and the norm is equivalent to the order unit norm.
\end{proposition}
\begin{proof}
(iii)$\Rightarrow$(i) is clear.

(i)$\Rightarrow$(ii): Let $\left(x_{\alpha}\right)_{\alpha\in A}\subseteq X$ be norm-null. Every quasi-subnet of this net is also norm-null, and so has an order bounded subnet (see \cite{bctv} for the concept of a quasi-subnet). According to \cite[Proposition 9.5 ]{bctv} $\left(x_{\alpha}\right)_{\alpha\in A}$ has an order bounded tail.\medskip

(ii)$\Rightarrow$(iii): Let $X\backslash \left\{0\right\}$ be considered a net indexed by itself, with the order $x\prec y$ if $\|x\|\ge \|y\|$. Clearly, this net is norm-null, and so it has an order bounded tail. Unpacking the latter observation yields $r>0$ such that $r\Bo_{X}$ is order bounded. Therefore, there is $u\in X_{+}$ such that $\Bo_{X}\subseteq \left[-u,u\right]\subseteq \|u\|\Bo_{X}$. This means that $u$ is a strong unit in $X$, and the norm of $X$ is equivalent to the norm generated by $u$.
\end{proof}

\begin{remark}
The presented results can be easily adapted to the setting of metrizable locally solid spaces. Indeed, the topology of  such a space is determined by a single Riesz pseudo-norm (see \cite[Definition 2.27]{ab0}). The only properties of the norm that we have used in the proofs were subadditivity, $\|nx\|\le n\|x\|$, for $x\in X$ and $n\in\N$ (which also follows from subadditivity), and $\|\frac{1}{n}x\|\to 0$, which follows from linearity of the topology.
\qed\end{remark}

We now compare the conditions considered in this section.

\begin{proposition}\label{compare}
For a normed lattice $X$, consider the following conditions:
\item[(i)] $X$ satisfies the conditions of Theorem \ref{rf} (i.e. is a Banach lattice);
\item[(ii)] $X$ satisfies the conditions of Proposition \ref{arf} (i.e. is majorizing in its norm completion);
\item[(iii)] $X$ satisfies the conditions of Proposition \ref{arfa} (i.e. every norm-null sequence in $X$ has an order bounded subsequence);
\item[(iv)] $X$ satisfies the conditions of Proposition \ref{rfa} (i.e. the norm on $X$ is equivalent to an order unit norm).\medskip

Then, (i),(iv)$\Rightarrow$(ii)$\Rightarrow$(iii), while (i) and (iv) are not comparable (and so (ii)$\not\Rightarrow$(iv) and (ii)$\not\Rightarrow$(i) ).
\end{proposition}
\begin{proof}
(i)$\Rightarrow$(ii) is trivial, while (iv)$\Rightarrow$(ii) is easy to see. For (iv)$\not\Rightarrow$(i) take a dense proper sublattice of $\Co\left(K\right)$, e.g. the space $F$ of piecewise affine functions inside $\Co\left[0,1\right]$, while for (i)$\not\Rightarrow$(iv) take $\ell_{1}$. In fact, an $\ell_{1}$ sum of a sequence of copies of $F$'s is a normed lattice which satisfies (ii), but neither (i) nor (iv). Finally, (ii)$\Rightarrow$(iii) follows from Proposition \ref{arf}.
\end{proof}

The only implication that has not been addressed by Proposition \ref{compare} is whether (ii)$\Rightarrow$(iii) is strict (this is also the content of \cite[Question 8.17]{bt}). The rest of the article is dedicated to partially answering this question. Recall that an \emph{ideal on $\N$} is a collection of subsets of $\N$ which is closed with respect to taking finite unions and subsets. An ideal is \emph{meager} if it is meager as a subset of $\left\{0,1\right\}^{\N}$. An ideal is a \emph{P-ideal} if it is a proper ideal which contains all finite sets, and whose image in $\Po\left(\N\right)\slash fin\left(\N\right)$ generates a P-set in $\beta \N\backslash \N$. Note $\Io$ is a P-ideal if and only if whenever $\left(A_{n}\right)_{n\in\N}\subseteq \Io$, there is $A\in\Io$ such that $A_{n}\backslash A$ is finite, for every $n\in\N$. Furthermore, $\Io$ is a non-meager P-ideal if and only if whenever $\left(A_{n}\right)_{n\in\N}\subseteq\Io$ is such that $A_{n}\subseteq\N\backslash\left\{1,...,n-1\right\}$, for every $n\in\N$, there is an infinite $M\subseteq \N$ such that $\bigcup\limits_{n\in M}A_{n}\in\Io$ (see \cite[Proposition 2.3]{jmps}). We will be using the latter characterization as the definition. Non-meager P-ideals exist under various set-theoretical assumptions (see \cite{jmps}), and in particular, they exist under the Continuum Hypothesis, but it is still an open problem whether their existence can be proven in ZFC (see also \cite{kmz} and \cite{hern} for more information). It turns out that this old set-theoretical question is equivalent to \cite[Question 8.17]{bt}.

\begin{theorem}
Existence of a non-meager P-ideal is equivalent to existence of a normed lattice such that every norm-null sequence has an order bounded subsequence, but which is not majorizing in its norm completion. In particular, such a normed lattice exists under the Continuum Hypothesis.
\end{theorem}
\begin{proof}
Necessity: We will view sequences as functions on $\N$. Let $\Io$ be a non-meager P-ideal. Let $G$ be the set of all bounded sequences whose support is an element of $\Io$, and let $H:=\left\{h\in\ell_{\8},~ 2^{k}h\left(k\right)\to 0\right\}$. Note that both $G$ and $H$ are ideals in $\ell_{\8}$. Hence, $X:=G+H$ is also an ideal in $\ell_{\8}$, and in particular, it is order complete. We will show that $\left(X,\|\cdot\|_{\8}\right)$ satisfies the required properties.

Let $e_{n}:=\1_{\left\{n\right\}}\in X$, for every $n\in\N$. We claim that $\left(\frac{1}{2^{n}}e_{n}\right)_{n\in\N}$ is not order bounded in $X$. Assume that $g\in G$ and $h\in H$. There is $n\in\N$ such that $h\left(k\right)<\frac{1}{2^{k}}$, for every $k\ge n$. Since $\supp~ g$ is not co-finite there is $k\ge n$ such that $k\notin\supp~ g$. Then, $\left[g+h\right]\left(k\right)=h\left(k\right)<\frac{1}{2^{k}}=\frac{1}{2^{k}}e_{k}\left(k\right)$, and so $g+h\not\ge \frac{1}{2^{k}}e_{k}$. It now follows from Proposition \ref{arf} that $X$ is not majorizing in its norm completion.\medskip

In order to establish the second property, it is enough to prove that if $\left(f_{n}\right)_{n\in\N}\subseteq X$ is such that $\|f_{n}\|_{\8}\le\frac{1}{4^{n}}$, for every $n\in\N$, then this sequence is order bounded. Without loss of generality we may also assume that $f_{n}\ge 0$, for every $n\in\N$.

Fix $n\in\N$. There is $g_{n}\in G$ such that $h_{n}:=f_{n}-g_{n}\in H_{+}$. Then, there is $N_{n}\in\N$ such that $N_{n}\ge n$ and $2^{k}h_{n}\left(k\right)<\frac{1}{2^{n}}$, for every $k\ge N_{n}$. Let $A_{n}:=\supp~ g_{n}\in\Io$, and let $B_{n}:=\left(A_{n}\cup\left[n,N_{n}\right]\right)\backslash\left[1,n-1\right] \in\Io$. Let $u_{n}:=f_{n}\1_{B_{n}}\in G$ and let $v_{n}:=f_{n}-u_{n}$. Also, define $w_{n}$ by $w_{n}\left(k\right):=2^{k}v_{n}\left(k\right)$, for $k\in\N$. Let us show that $v_{n}\in H$ (which is equivalent to $w_{n}\in c_{0}$) and that $2^{k}v_{n}\left(k\right)\le 2^{-n}$, for every $k\in\N$ (which is equivalent to $\|w_{n}\|_{\8}\le 2^{-n}$).

First, if $k\in B_{n}$, then $u_{n}\left(k\right)=f_{n}\left(k\right)$, and so $v_{n}\left(k\right)=0$. If $k\notin B_{n}$, then either $k<n$, or $k>N_{n}$. In the former case, we have $v_{n}\left(k\right)=f_{n}\left(k\right)\le 4^{-n}$, and so $2^{k}v_{n}\left(k\right)\le 2^{k-2n}\le 2^{-n}$. In the latter case we also have that $k\notin A_{n}$, hence $g_{n}\left(k\right)=0$, and so $v_{n}\left(k\right)=f_{n}\left(k\right)=h_{n}\left(k\right)\le 2^{-k-n}$, implying $2^{k}v_{n}\left(k\right)\le 2^{-n}$. Also, for $k>N_{n}$ either $k\in A_{n}$, and then $v_{n}\left(k\right)=0$, or $k\notin A_{n}$, and then $v_{n}\left(k\right)= h_{n}\left(k\right)$. In both cases $2^{k}v_{n}\left(k\right)\le 2^{k}h_{n}\left(k\right)\to 0$, so that $v_{n}\in H$.\medskip

By the assumption about $\Io$, there is an infinite $M\subseteq\N$ such that $B:=\bigcup\limits_{n\in M}B_{n}\in\Io$. Since $\|u_{n}\|_{\8}\le \|f_{n}\|_{\8}\le 4^{-n}$, for every $n\in\N$, there is $u:=\sum\limits_{n\in M}u_{n}$ in $\ell_{\8}$. Moreover, $\supp~ u\subseteq B$, and so $u\in G$. Clearly, also $u\ge u_{n}$, for every $n\in M$.\medskip

Since $\|v_{n}\|_{\8}\le \|f_{n}\|_{\8}\le 4^{-n}$, for every $n\in M$, there is $v=\sum\limits_{n\in M}v_{n}$ in $\ell_{\8}$. On the other hand, as $w_{n}\in c_{0}$ and $\|w_{n}\|_{\8}\le 2^{-n}$, for every $n\in M$, there is $w=\sum\limits_{n\in M}w_{n}$ in $c_{0}$. Note that for every $k\in\N$ we have that $w\left(k\right)=\sum\limits_{n\in M}w_{n}\left(k\right)=\sum\limits_{n\in M}2^{k}v_{n}\left(k\right)=2^{k}v\left(k\right)$, hence, $2^{k}v\left(k\right)=w\left(k\right)\to 0$, implying that $v\in H$. Clearly, also $v\ge v_{n}$, for every $n\in M$. Thus, $u+v\in X$ and $u+v\ge u_{n}+v_{n}=f_{n}$, for every $n\in M$.\bigskip

Sufficiency: Assume that $X$ is a normed lattice such that every norm-null sequence contains an order bounded subsequence, but it is not majorizing in its norm completion. A careful look at the proof of the implication (iv)$\Rightarrow$(iii) in Proposition \ref{arf} reveals that there is a positive sequence $\left(e_{n}\right)_{n\in\N}$ with $\|e_{n}\|\le 2^{-n}$, for all $n\in\N$, which is not order bounded. We will now construct a non-meager P-ideal on $\N$.\medskip

For $A\subseteq \N$ let $\mu\left(A\right):=\bigwedge\left\{\|x\|,~ x\in X_{+},~ \forall i\in A:~ x\ge e_{i}\right\}$. Let us show that $\mu$ is a submeasure on $\Po\left(\N\right)$ (i.e. a subadditive monotone function from $\Po\left(\N\right)$ into $\left[0,\8\right]$ which vanishes at $\varnothing$). Clearly, $\mu\left(\varnothing\right)=0$ and if $A\subseteq B$, then $\mu\left(A\right)\le \mu\left(B\right)$. Assume that $A,B\subseteq \N$ are such that $\mu\left(A\right),\mu\left(B\right)<\8$, and let $\varepsilon>0$. There are $x,y\in X_{+}$ such that $\|x\|\le\mu\left(A\right)+\varepsilon$, $\|y\|\le\mu\left(B\right)+\varepsilon$, $x\ge e_{i}$ and $y\ge e_{j}$, for every $i\in A$ and $j\in B$. Then, it is easy to see that $x+y\ge e_{i}$, for every $i\in A\cup B$, and so $\mu\left(A\cup B\right)\le \|x+y\|\le \|x\|+\|y\|\le \mu\left(A\right)+\mu\left(B\right)+2\varepsilon$. As $\varepsilon$ was arbitrary, subadditivity follows.\medskip

Note that if $A\subseteq\N$ is finite, then $\mu\left(A\right)\le \left\|\sum\limits_{i\in A}e_{i}\right\|\le \sum\limits_{i\in A}\left\|e_{i}\right\|\le 2^{1-\min A}$. On the other hand, $\mu\left(\N\right)=\8$.

Let $\Jo:=\left\{A\subseteq\N,~ \mu\left(A\right)<\8\right\}$, which is clearly an ideal in $\Po\left(\N\right)$, and let $\Io$ be the order continuous part of $\Jo$, i.e. $\Io:=\left\{A\subseteq\N,~ \mu\left(A^{n}\right)\to 0\right\}$, where $A^{n}:=A\cap\left[n,\8\right)$. It is easy to see that $\Io$ is also an ideal, $\N\notin\Io$, but every finite set belongs to $\Io$. Moreover, if $A\in\Io$, then $\mu\left(A\right)\le 2^{1-n}$, where $n:=\min A$. Indeed, for every $\varepsilon>0$ there is $N\in\N$ such that $\mu\left(A^{N}\right)<\varepsilon$; then $\mu\left(A\right)\le \mu\left(\left[n,N\right]\right)+\mu\left(A^{N}\right)\le 2^{1-n}+\varepsilon$.\medskip

Let us show that $\Io$ is closed in $\Jo$ with respect to the pseudometric $\left(A,B\right)\mapsto\mu\left(A\vartriangle B\right)$. Let $A\in\Jo$ be such that for every $\varepsilon>0$ there is $B\in\Io$ such that $\mu\left(A\vartriangle B\right)<\varepsilon$. There is $N\in\N$ such that $\mu\left(B^{n}\right)<\varepsilon$, for every $n\ge N$. Then, since $A^{n}\subseteq B^{n}\cup\left(A\vartriangle B\right)$, it follows that $\mu\left(A^{n}\right)\le \mu\left(B^{n}\right) + \mu\left(A\vartriangle B\right)\le 2\varepsilon$ for every $n\ge N$. As $\varepsilon$ was arbitrary, we conclude that $A\in\Io$.\medskip

Now let $\left(A_{n}\right)\subseteq\Io$ be such that $A_{n}\subseteq\N\backslash\left\{1,...,n-1\right\}$, for every $n\in\N$. It follows from the previous step that for every $n\in\N$ we have $\mu\left(A_{n}\right)\le 2^{1-n}$, and so there is $x_{n}\in X_{+}$ with $\|x_{n}\|\le 2^{2-n}$ such that $x_{n}\ge e_{i}$, for every $i\in A_{n}$. Then, $nx_{n}\to 0$, and so by assumption there is an infinite $M\subseteq\N$ and $x\in X$ such that $nx_{n}\le x$, for every $n\in M$. It is enough to show that $\bigcup\limits_{m\in M}A_{m}\in\Io$.

For $n\in\N$ let $B_{n}:=\bigcup\limits_{m\in M\cap\left[n,\8\right)}A_{m}$; note that if $i\in B_{n}$, then there is $m\in M\cap\left[n,\8\right)$ such that $i\in A_{m}$, and so $e_{i}\le x_{m}\le \frac{1}{m}x \le \frac{1}{n}x$. It follows that $\mu\left(B_{n}\right)\le \frac{1}{n}\|x\|$, for every $n\in\N$. In particular $\bigcup\limits_{m\in M}A_{m}=B_{1}\in\Jo$. Also,  $B_{1}\vartriangle \bigcup\limits_{m\in M\cap\left[1,n-1\right]}A_{m}\subseteq B_{n}$ implies that $\Io\ni \bigcup\limits_{m\in M\cap\left[1,n-1\right]}A_{m}\to B_{1}$. As $\Io$ is closed in $\Jo$ we conclude that $B_{1}\in\Io$.
\end{proof}



\end{document}